\documentclass{amsart}
\usepackage{amscd}
\usepackage{amssymb}
\input xypic
\xyoption{all}
\newdir{ >}{{}*!/-5pt/@{>}}

\def\id{\operatorname{id}}
\def\cA{\mathcal A}

\def\cC{\mathcal C}

\def\ob{\mathrm{ob}}

\def\obc{\ob\cC}
\def\cD{\mathcal D}

\def\cF{\mathcal F}

\def\cS{\mathcal S}

\newcommand{\Z}{\mathbb{Z}}
\newcommand{\N}{\mathbb{N}}
\newcommand{\im}{\operatorname{im}}

\def\coker{\operatorname{coker}}
\def\colim{\operatornamewithlimits{colim}}

\numberwithin{equation}{section}
\theoremstyle{plain}
\newtheorem{thm}[equation]{Theorem}
\newtheorem{cor}[equation]{Corollary}
\newtheorem{lem}[equation]{Lemma}
\newtheorem{prop}[equation]{Proposition}

\newcommand{\comment}[1]{}  
\theoremstyle{definition}
\newtheorem{defn}[equation]{Definition}
\theoremstyle{remark}
\newtheorem{rem}[equation]{Remark}
\newtheorem{ex}[equation]{Example}

\begin{document}

\bibliographystyle{plain}

\title{On the K-theory of $\mathbb{Z}$-categories.}
\author{Eugenia Ellis}
\email{eellis@fing.edu.uy}
\address{IMERL, Facultad de Ingenier\'\i a, Universidad de la Rep\'ublica \\
Julio Herrera y Reissig 565, 11.300, Montevideo, Uruguay}
\author{Rafael Parra}
\email{rparra@fing.edu.uy}
\address{IMERL, Facultad de Ingenier\'\i a, Universidad de la Rep\'ublica \\
Julio Herrera y Reissig 565, 11.300, Montevideo, Uruguay}

\begin{abstract}
We establish connections between the concepts of Noetherian, regular coherent, and regular $n$-coherent categories for $\Z$-linear categories with finitely many objects and the corresponding notions for unital rings. 
These connections enable us to obtain a negative $K$-theory vanishing result, a fundamental theorem, and a homotopy invariance result for the $K$-theory of $\Z$-linear categories. 
\end{abstract}

\maketitle

\section{Introduction}

Let $R$ be an associative ring with a unit.  The fundamental theorem in K-theory also known as the Bass-Heller-Swan theorem, expresses the K-groups of $R[t,t^{-1}]$ in terms of the K-groups and Nil-groups of $R$
$$
K_{i}(R[t,t^{-1}])\simeq K_{i-1}(R) \oplus K_{i}(R) \oplus \operatorname{Nil}_{i-1}(R)\oplus \operatorname{Nil}_{i-1}(R).
$$
The groups $\operatorname{Nil}_i(R)$ for $i\in\mathbb{Z}$, and the $K$-groups $K_i(R)$ for $i < 0$, are known to vanish when $R$ is right regular (i.e. right Noetherian and right regular coherent).  Swan \cite{Swan} proved that $\operatorname{Nil}_i(R)$ also vanishes when $R$ is right regular coherent and $i\geq 0$, using Quillen's resolution and devissage theorems as the main tools. In \cite{ep}, we extended the study to $n$-coherent rings, where $n\geq 0$ (here 1-coherent ring is the same as coherent ring, and 0-coherent ring is the same as Noetherian ring). We derived a new expression for $\operatorname{Nil}_i(R)$ for a $n$-regular and $n$-coherent ring $R$, but its vanishing status remains unknown. Our current focus is on exploring various methods for computing these groups.

The algebraic K-theory of a ring with a unit can be generalized to categories that have additional structure, and even to non-unital rings. In the context of K-theory, it is often more convenient to use additive categories instead of rings. With this motivation in mind, Bartels and L\"uck extended the notions of regularity and regular coherence to additive categories. In \cite{bl} they proved the following result:

\begin{thm}\cite[Corollary 12.2]{bl} Let  $\mathcal{C}$ be an additive category which is regular. Then $K_i( \mathcal{C})=0$ for all $i\leq -1$.
\end{thm}

The focus of this paper is to extend the notion of regular $n$-coherence and some vanishing results in K-theory from rings to $\Z$-linear categories. Let $\cC$ be a $\Z$-linear category. We define a right $\cC$-module as a contravariant $\Z$-linear functor $F: \cC^{op} \rightarrow{\operatorname{Ab}}$.  We denote the category of right $\mathcal{C}$-modules as $\operatorname{Fun}(\mathcal{C}^{op}, \operatorname{Ab})$. Using the Yoneda lemma, we embed $\cC$ into $\operatorname{Fun}(\cC^{op}, \operatorname{Ab})$ with the purpose of using homological constructions in $\operatorname{Fun}(\cC^{op}, \operatorname{Ab})$  which a priori make no sense in  $\cC$. The finiteness conditions for $\operatorname{Fun}(\mathcal{C}^{op}, \operatorname{Ab})$ are defined in \cite{bgp}.  As the category of $R$-modules, $\operatorname{Fun}(\cC^{op}, \operatorname{Ab})$ is a Grothendieck category with a generating set of finitely generated projective objects. A right $\cC$-module $F$ is said to be of type $\mathcal{FP}_{n}$ if and only if there exists an exact sequence 
$$ P_{n} \rightarrow \cdots \rightarrow P_{1} \rightarrow P_{0} \xrightarrow \ F \rightarrow 0$$
where $P_{i}$ is a finitely generated and projective right $\cC$-module for every $0\leq{i}\leq{n}$. A right $\cC$-module $F$ is of type $\mathcal{FP}_{\infty}$ if it is of type $\mathcal{FP}_{n}$ for all $n\geq{0}.$

We say that $\cC$ is right $n$-coherent if the category $\operatorname{Fun}(\cC^{op}, \operatorname{Ab})$ is $n$-coherent in the sense of \cite[Definition 4.6]{bgp}. In other words, $\cC$ is right $n$-coherent if and only if the $\cC$-modules of type $\mathcal{FP}_{n}$ in $\operatorname{Fun}(\cC^{op}, \operatorname{Ab})$  coincide with those of type  $\mathcal{FP}_{\infty}$. We say that $\cC$ is right {\emph{regular $n$-coherent}} if $\cC$ is right $n$-coherent and every $\cC$-module $F$  of type $\mathcal{FP}_{n}$ has finite projective dimension.  In Proposition \ref{eq}, we prove that this homological property of $\mathcal{C}$ also holds for $\mathcal{C}_{\oplus}$.
 
Let $n\geq 1$ and $f:x\rightarrow y$ be a morphism in $\cC$. Following \cite{BOPM}, we say that $f$ has a \emph{pseudo $n$-kernel}  if there exists a chain of morphisms: 
$$ x_n \xrightarrow{f_{n}} x_{n-1}\xrightarrow{f_{n-1}} x_{n-2}\rightarrow\cdots\xrightarrow{f_{2}} x_1\xrightarrow{f_{1}} x\xrightarrow{f} y $$
such that the following sequence of abelian groups is exact:
$$ \hom_\cC(-, x_n)\xrightarrow{f_{n*}} \cdots\rightarrow \hom_\cC( -,x_1)\xrightarrow{f_{1*}} \hom_\cC( -,x)\xrightarrow{f_*} \hom_\cC( -,y).$$
In Proposition \ref{char}, we establish necessary and sufficient intrinsic conditions on $\mathcal{C}$ for it to be right regular $n$-coherent. Specifically, we demonstrate that an additive category $\mathcal{C}$ is right regular $n$-coherent if and only if the following conditions hold:
\begin{enumerate}
\item[i)] Every morphism in $\cC$ with a pseudo $(n-1)$-kernel has a pseudo $n$-kernel. 
\item[ii)] For every morphism $f:x\rightarrow y$ in $\cC$ with a pseudo $\infty$-kernel, there exist $k\in\mathbb{N}$ and a morphism $\alpha:x_{k-1}\rightarrow x_{k-1}$ such that the following diagram commutes:
$$\xymatrix{
x_k\ar[rd]^{0} \ar[r]^{f_{k}} & x_{k-1}\ar[d]^{\alpha} \ar[r]^{f_{k-1}} & x_{k-2} \ar[r]^{f_{k-2}}& \cdots
&\cdots \ar[r]^{f_{2}}& x_{1} \ar[r]^{f_{1}}& x \ar[r]^{f}& y \\
& x_{k-1}\ar[ru]_{f_{k-1}}}$$\\
\end{enumerate}

The algebraic K-theory of a $\Z$-linear category $\mathcal{C}$ is defined using the non-connective spectrum $\bold{K}^{\infty}(\mathcal{C}_{\oplus})$, which was introduced in \cite{pw}. Furthermore, $\mathcal{C}$ is associated with a ring defined as $$\mathcal{A}(\mathcal{C})=\bigoplus_{a,b\in\obc}\hom_\cC(a,b),$$ where $\obc$ denotes the objects of $\mathcal{C}$. The multiplication and addition in $\mathcal{A}(\mathcal{C})$ are defined naturally, resulting in a ring with local units.  It is important to note that $\mathcal{A}(\mathcal{C})$ is unital only when $\obc$ is finite.   Furthermore,  it is worth mentioning that there exists a weak equivalence between the spectrum of the algebraic K-theory of $\mathcal{C}$ and the spectrum of the algebraic K-theory of $\mathcal{A}(\mathcal{C})$, as shown in \cite[Sec. 4.2]{ce}. Therefore, the K-theory groups of $\mathcal{C}$ and $\mathcal{A}(\mathcal{C})$ coincide.

In Section \ref{sec3}, we compare the notions of Noetherianity, regular coherence, and regular $n$-coherence of $\mathcal{A}(\mathcal{C})$ with the corresponding notions for $\mathcal{C}$ when $\mathcal{C}$ is a $\mathbb{Z}$-linear category with finitely many objects. This comparison allows us to establish a relationship between the properties of $\mathcal{C}$ and the properties of $\mathcal{A}(\mathcal{C})$. It is important to note that although some of the results we have used do not require the condition $\mathcal{C}$ having finitely many objects, this condition is necessary to guarantee that the ring $\mathcal{A}(\mathcal{C})$ has a unit.
A $\mathbb{Z}$-linear category $\mathcal{C}$ is right regular $n$-coherent if and only if the additive category $\mathcal{C}_{\oplus}$ associated with $\mathcal{C}$ has this property, as stated in Proposition \ref{eq}. The reason for working with $\mathbb{Z}$-linear categories instead of additive categories is that the ring $\mathcal{A}(\mathcal{C}_{\oplus})$ does not have a unit due to the fact that $\mathcal{C}_{\oplus}$ has infinitely many objects. By considering $\mathbb{Z}$-linear categories, we are able to address this issue and ensure the existence of a unit for the corresponding ring.

We see in Proposition \ref{propeq} that the category $\operatorname{Fun}(\cC^{op}, \operatorname{Ab})$ is equivalent  to $\operatorname{Mod}$-$\cA(\cC)$, where $\operatorname{Mod}$-$\cA(\cC)$ denotes the category of unital right modules over $\cA(\cC)$. Furthermore, in Proposition \ref{catring}, we prove that $\cC$ is  a $\Z$-linear category with finitely many objects  then a $\cC$ is right Noetherian ($n$-coherent or regular $n$-coherent) if and only if $\cA(\cC)$ is (strong $n$-coherent or $n$-regular and strong $n$-coherent). We use Proposition \ref{catring} and results for rings in order to obtain information about the K-theory of a $\Z$-linear category. These results are not completely original but the way to obtain them is.

In Section \ref{sec4}, we prove that if $\mathcal{D}=\cC$, $\mathcal{D}=\cC_{\oplus}$ or $\mathcal{D}=\operatorname{colim}_{f\in F}\cC_{f}$ with $\cC$ and $\cC_{f}$ regular $\Z$-linear categories with finitely many objects, then 
$K_{i}(\mathcal{D})=0$ $\forall i <0$.
We also prove that if $\mathcal{D}=\cC$, $\mathcal{D}=\cC_{\oplus}$ or $\mathcal{D}=\operatorname{colim}_{f\in F}\cC_{f}$ with $\cC$ and $\cC_{f}$ regular coherent $\Z$-linear categories with finitely many objects, then 
$K_{-1}(\mathcal{D})=0$, 
$$
K_{i}(\mathcal{D})\simeq K_{i}(\mathcal{D}[t])\quad \mbox{and} \quad K_{i+1}(\mathcal{D}[t,t^{-1}]) \simeq K_{i+1}(\mathcal{D})\oplus K_{i}(\mathcal{D}) \quad \forall i \geq 0.
$$
In Proposition \ref{generaln} we obtain a generalization of \cite[Thm 3.2]{ep}.

\section{Modules over 
$\Z$-linear categories}

 A {\emph{$\Z$-linear category}} is a category $\cC$ such that for every two objects $a, b \in\cC$, the set of morphisms $\hom_\cC(a,b)$ is an abelian group, and for any other object $c \in\cC$, the composition
$$ \hom_\cC(b,c)\times\hom_\cC(a,b)\rightarrow\hom_\cC(a,c)$$ 
is a bilinear map.
Throughout this paper, we assume that $\Z$-linear categories $\cC$ are  small, i.e.  the collection of objects is a set. 
A $\Z$-linear category is  {\emph{additive}} if it has an initial object and finite products. We consider the free additive category $\cC_{\oplus}$ as follow. The objects of $\cC_{\oplus}$ are finite tuples of objects in $\cC$. A morphism from $\textbf{a}=(a_{1}, \cdots,a_{k})$ to $\textbf{c}=(c_{1}, \cdots,c_{m})$ for $a_{i}, c_{j} \in\cC$ is given by $m\times{k}$ matrix of morphisms in $\cC$ (the composition is given by the usual row-by-column multiplication of matrices), 
\begin{itemize}
\item[-] $\obc_{\oplus}=\lbrace{(c_{1}, \cdots,c_{k}): c_{i}\in\cC, k\in \mathbb{N}}\rbrace $
\item[-] $\operatorname{hom}_{\cC_{\oplus}}(\textbf{a},\textbf{c})= \prod_{i=1}^{k}\prod_{j=1}^{m}\hom_\cC(a_{i},c_{j}).$ 
\end{itemize}
There is an obvious embedding  $ \cC \rightarrow \cC_{\oplus}$ which maps objects and morphisms to their associated $1$-tuple. If  $\cC$ is a $\Z$-linear category then $\cC_{\oplus}$ is a small additive category.

The {\emph{idempotent completion}} $\operatorname{Idem}(\cC_{\oplus})$ of $\cC_{\oplus}$ is defined to be the following small additive category.
\begin{itemize}
\item[-] $\ob(\operatorname{Idem}(\cC_{\oplus}))=\{({\bf{c}},p):\ c\in \obc_{\oplus},\ p:{\bf{c}}\rightarrow{\bf{c}}$  such that $p^{2}=p\}$
\item[-] $\hom_{\operatorname{Idem}(\cC_{\oplus})}(({\bf{c}}_{1},p_{1}),({\bf{c}}_{2},p_{2}))=\lbrace w:{\bf{c}}_{1}\rightarrow{\bf{c}}_{2}$  such that  $w=p_{2}wp_{1} \}$.
\end{itemize}
By construction $\cC\simeq\cC_{\oplus}$ if $\cC$ is additive and $\cC_{\oplus}\simeq{\operatorname{Idem}}(\cC_{\oplus})$ if idempotents split in $\cC_{\oplus}.$ Recall the additive category $\cC_{\oplus}$ is  equivalent to Idem$(\cC_{\oplus})$ if and only if every idempotent has a kernel.

\begin{ex}\label{rincatr}
Given a ring $R$, consider $\cC=\underline{R}$ the category which has one object $\star$ and $\hom_{\cC}(\star, \star)=R$. The multiplication on $R$ gives the composition on $\underline{R}$. The category $\cC_{\oplus}$ is the category whose objects are natural numbers $m>0$ and the morphisms are the matrices with coefficients in $R$, $\hom_{\cC_{\oplus}}(m,n)=M_{n\times m}(R)$.
\end{ex}

\begin{ex} Let $R$ be an associative ring with unity. If  $\cC$ is the category of finitely generated free $R$-modules, then  Idem($\cC$) is equivalent to the category of finitely generated projective $R$-modules.
\end{ex}

\subsection{Pseudo $n$-kernels and pseudo $n$-cokernels}

Given a $\Z$-linear category $\cC$ we recall that a \emph{pseudo kernel} of a morphism $f:x\rightarrow y$ in $\cC$ is a morphism   $g:k\rightarrow x$ with $f\circ{g}=0$, such that for any morphism  $h:c\rightarrow x$ with $f\circ{h}=0,$ there exists $t:c\rightarrow k$ with $g\circ{t}=h$. Equivalently, a morphism $g:k\rightarrow x$ in $\cC$ is said to be a pseudo kernel of $f$ if, for any $c\in\ob \cC$, the following sequence of abelian groups is exact
 $$ \hom_\cC(c,k)\rightarrow \hom_\cC(c,x)\rightarrow \hom_\cC(c,y).$$  
Pseudo-kernels have been introduced by Freyd \cite{Frey} as weak kernels. \emph{Pseudo-cokernels} are pseudo kernels in $\cC^{op}$. By \cite[Corollary 1.1]{PSV}  the categories $\cC$, $\cC_{\oplus}$ and $\operatorname{Idem}(\cC_{\oplus})$ all have pseudo kernels or they don´t. Let us remark that any triangulated or abelian category has pseudo-kernels and pseudo-cokernels. 

 Let $n\geq 1$ and $f:x\rightarrow y$ be a morphism in $\cC$. Following \cite{BOPM}, we say that $f$ has a \emph{pseudo $n$-kernel}  if there exists a chain of morphisms 
$$ x_n \xrightarrow{f_{n}} x_{n-1}\xrightarrow{f_{n-1}} x_{n-2}\rightarrow\cdots\xrightarrow{f_{2}} x_1\xrightarrow{f_{1}} x\xrightarrow{f} y $$
such that the following sequence of abelian groups is exact
$$ \hom_\cC(-, x_n)\xrightarrow{f_{n*}} \cdots\rightarrow \hom_\cC( -,x_1)\xrightarrow{f_{1*}} \hom_\cC( -,x)\xrightarrow{f_*} \hom_\cC( -,y).$$
We denote the pseudo $n$-kernel by $(f_n,f_{n-1},\cdots,f_{1})$. The case $n=1$ gives us the classic pseudo-kernels. For convenience, we let $x_0:=x$. Furthermore, any morphism $f$ in $\cC$ will be assumed to be a pseudo $0$-kernel of itself. We say that $f$ has a \emph{ pseudo $\infty$-kernel}  if there exists a chain of morphisms 
$$\cdots\rightarrow x_{n+1} \xrightarrow{f_{n+1}} x_n \xrightarrow{f_{n}} x_{n-1}\rightarrow\cdots\xrightarrow{f_{2}} x_1\xrightarrow{f_{1}} x\xrightarrow{f} y $$
such that the following sequence of abelian groups is exact
$$\cdots \rightarrow \hom_\cC(-, x_{n+1})\xrightarrow{f_{n+1*}} \hom_\cC(-, x_n)\xrightarrow{f_{n*}} \cdots\xrightarrow{f_{1*}} \hom_\cC( -,x)\xrightarrow{f_*} \hom_\cC( -,y).$$
Pseudo $n$-cokernels are defined as pseudo $n$-kernels in $\cC^{op}$. 

\subsection{Categories of $\Z$-linear functors}

The category of abelian groups will be denoted by $\operatorname{Ab}$. 
For any $\Z$-linear category $\cC,$ we define a {\emph{left $\cC$-module}} as a $\Z$-linear functor $F: \cC \rightarrow{\operatorname{Ab}}$. We consider natural transformations as morphisms of $\cC$-modules. Define a {\emph{right $\cC$-module}} as a $\Z$-linear functor $F: \cC^{op} \rightarrow{\operatorname{Ab}}$. Recall that a $\Z$-linear functor $F: \cC^{op} \rightarrow \operatorname{Ab}$ satisfies that $F(f+g) = F(f)+F(g)$ where $f,g\in \hom_{\cC^{op}}(x,y)$.  In these categories limits and colimits of functors are defined objectwise.  Denote by $\operatorname{Fun}(\cC^{op}, \operatorname{Ab})$ the category of right $\cC$-modules. This category is cocomplete  and abelian.
If $c$ is an object of $\cC$ then there is the corresponding representable functor $\hom_{\cC}(-,c): \cC^{op} \rightarrow{\operatorname{Ab}}$. 
\begin{lem}(Yoneda Lemma) Let $\cC$ be any $\Z$-linear category. Take $c\in \cC$ and $F$ a right $\cC$-module. Then there is a natural identification 
$$\hom_{\operatorname{Fun}(\cC^{op}, \operatorname{Ab})}(\hom_{\cC}(-,c),F(-))\cong F(c).$$\end{lem}
By Yoneda Lemma, the family $\lbrace{\hom_{\cC}(-,c)}\rbrace_{c\in\cC}$ is a generating set of finitely generated projective modules in $\operatorname{Fun}(\cC^{op}, \operatorname{Ab})$. A right $\cC$-module $M$ is {\emph{free}}  if it is isomorphic to 
$\bigoplus_{i\in I}\hom_{\cC}(-,a_{i})$. It is free and finitely generated if $I$ is finite.\\

Let $R$ be a ring and $\underline{R}$ be the $\Z$-linear category defined in Example \ref{rincatr}. Note that
$$
\mbox{$R$-$\operatorname{Mod}$} \cong \operatorname{Fun}(\underline{R}, \operatorname{Ab})$$
$$\mbox{$\operatorname{Mod}$-$R$} \cong \operatorname{Fun}(\underline{R}^{op}, \operatorname{Ab}).$$

\subsection{Finitely $n$-presented objects and $n$-coherent categories}

Let $n\geq{1}$ be a positive integer.  According to \cite[Definition 2.1]{bgp} 
 a right $\cC$-module $F$ is said to be \emph{finitely $n$-presented} or \emph{of type $\mathcal{FP}_{n}$} if the functors $\operatorname{Ext}^{i}_{\operatorname{Fun}(\cC^{op}, \operatorname{Ab})}(F,-)$ preserves direct limits for all $0\leq{i}\leq{n-1}.$  Denote by $\mathcal{FP}_{0}$ to the set of finitely generated objects. Then, a right $\cC$-module $M$ is of type $\mathcal{FP}_{0}$  if there exists a collection of objects $\lbrace{c_{j}: j\in{J}}\rbrace$ in $\cC$ for some finite set $J$ and an epimorphism $\bigoplus_{j\in{J}}\hom_\cC(-,c_{j})\rightarrow M$. Furthermore, a right $\cC$-module $F$ is said to be of type $\mathcal{FP}_{\infty}$ if it is of type $\mathcal{FP}_{n}$ for all $n\geq{0}.$\\

Recall that a {\emph{Grothendieck category}} is a cocomplete abelian category, with a generating set and with exact direct limits. A Grothendieck category is \emph{locally finitely generated (presented)} if it has a set of finitely generated (presented) generators. In other words,  each object is a direct union (limit) of finitely generated (presented) objects.  A Grothendieck category is \emph{locally type $\mathcal{FP}_{n}$} \cite[Definition 2.3]{bgp}, if it has a generating set consisting of objects of type $\mathcal{FP}_{n}$. 

 According to \cite[Example 3.2]{gil} any finitely generated projective right $\cC$-module is of type $\mathcal{FP}_{n}$ for all $n\geq{0}$. Then, the functor category $\operatorname{Fun}(\cC^{op} \operatorname{Ab})$ is a locally type $\mathcal{FP}_{\infty}$ Grothendieck category. Therefore, by the \cite[Corollary 2.14]{bgp}, a right $\cC$-module $F$ is of type $\mathcal{FP}_{n}$ if and only if there exists an exact sequence 
$$ P_{n} \rightarrow \cdots \rightarrow P_{1} \rightarrow P_{0} \xrightarrow \ F \rightarrow 0$$
where $P_{i}$ is finitely generated and projective right $\cC$-module for every $0\leq{i}\leq{n}.$

Recall from \cite[Definition 4.1]{bgp} that a right $\cC$-module $F$ is \emph{$n$-coherent} if satisfies the following conditions:
\begin{enumerate}
\item[(1)] $F$ is of type $\mathcal{FP}_{n}$.
\item[(2)]  If $S$ is a subobject of $F$ that is of type $\mathcal{FP}_{n-1}$ then $S$ is also of type $\mathcal{FP}_{n}$.
\end{enumerate}

\begin{defn} Let $\cC$ be a $\Z$-linear category  and $n\geq 0$. We say that $\cC$ is right (left) \emph{$n$-coherent} if every right (left) $\cC$-module $F$ of type $\mathcal{FP}_{n}$ is $n$-coherent. 
\end{defn}
Note that the $\Z$-linear category $\cC$ is right $n$-coherent if the category $\operatorname{Fun}(\cC^{op}, \operatorname{Ab})$ is $n$-coherent as a Grothendieck category in the sense of \cite[Definition 4.6]{bgp}. 
Thus by \cite[Theorem 4.7]{bgp},  $\cC$ is right $n$-coherent if and only if the $\cC$-modules of type $\mathcal{FP}_{n}$ coincide with the $\cC$-modules of type  $\mathcal{FP}_{\infty}$. 

In particular, an additive category $\cC$ is Noetherian as defined in \cite[Definition 5.2]{bl} if and only if it is right $0$-coherent.\footnote{In \cite{bl}, the word right is omitted, but we have chosen to include it in our notation.} 
Moreover, if $1\leq n\leq\infty$ and $\cC$ is any small additive category, then the following conditions are equivalent, as shown in \cite[Proposition 5.4]{BOPM}:
\begin{enumerate}
\item[(1)] $\cC$ is right $n$-coherent. 
\item[(2)] If a morphism in $\cC$ has a pseudo $(n-1)$-kernel, then it has a pseudo $n$-kernel. 
\end{enumerate}

\begin{defn} Let $\cC$ be a $\Z$-linear category and $n\geq 0$. We say that $\cC$ is \emph{right regular n-coherent}  if it satisfies the following conditions:
\begin{itemize}
\item[(1)] $\cC$ is right $n$-coherent.
\item[(2)] Every right $\cC$-module $F$ of type $\mathcal{FP}_{n}$ has a projective dimension. 
\end{itemize} 
\end{defn}
Let $\cC$ be a small additive category. Then, according to \cite[Definition 5.2]{bl} $\cC$ is regular coherent if and only if it is right regular $1$-coherent.

\begin{ex} Let $\cC$ be a small additive category and  $n\geq 1$.
\begin{enumerate}
\item[I.]\textbf{Additive category with  kernels.} By a result due to Auslander \cite[Theorem 2.2.b]{aus} a small additive category $\cC$ with  kernels is $1$-coherent and every $\cC$-module of type $\mathcal{FP}_{1}$ has projective dimension at most 2. Then $\cC$ is right regular $1$-coherent.
\item[II.]\textbf{Von Neumann regular categories.}  We recall that $\cC$ is called von Neumann regular if for any morphism $f:a\rightarrow b$ in $\cC$ there exists a morphism $g:b\rightarrow a$ such that $fgf=f$. By \cite[Corollary 8.1.3]{Apostol} $\cC$ is right regular $1$-coherent.
\item[III.]\textbf{Locally finitely presented categories.} An object $c\in\cC$ is finitely presented if the functor  $\hom_\cC(c,-)$ preserves direct limits. The category $\cC$ is locally finitely presented if every directed system of objects and morphisms has a direct limit, the class of finitely presented objects of $\cC$ is skeletally small and every object of $\cC$ is the direct limit of finitely presented objects.  Then by \cite[Lemma 2.2]{Garcia}, every locally finitely presented category is left $1$-coherent.  
\item[IV.]\textbf{n-hereditary categories}.  Suppose the  following two conditions hold in $\cC$:
\begin{enumerate}
    \item Every morphism in $\cC$ with a pseudo $(n-1)$-kernel has a pseudo $n$-kernel.
    \item For every morphism $f:x\rightarrow y$ in $\cC$ with pseudo $n$-kernel $(f_{n},\cdots,f_1)$, there exists an endomorphism $\alpha:x_{n-1}\rightarrow x_{n-1}$ making the following diagram commute:
$$\xymatrix{
x_k\ar[rd]^{0} \ar[r]^{f_{k}} & x_{k-1}\ar[d]^{\alpha} \ar[r]^{f_{k-1}} & x_{k-2} \ar[r]^{f_{k-2}}& \cdots
&\cdots \ar[r]^{f_{2}}& x_{1} \ar[r]^{f_{1}}& x \ar[r]^{f}& y \\
& x_{k-1}\ar[ru]_{f_{k-1}}}$$
\end{enumerate}
By \cite[Theorem 5.5]{BOPM}, $\cC$ is right $n$-coherent and every $\cC$-module of type $\mathcal{FP}_{n}$  has projective dimension less than or equal $1.$ Therefore, $\cC$ is right regular $n$-coherent.
\end{enumerate}
\end{ex}
Due to \cite[Lemma 1.1, 1.2]{PSV} we have the following equivalences of categories
 $$\operatorname{Fun}(\cC^{op}, \operatorname{Ab})\simeq \operatorname{Fun}(\cC^{op}_{\oplus}, \operatorname{Ab})\simeq \operatorname{Fun}(\operatorname{Idem}(\cC^{op}_{\oplus}), \operatorname{Ab})$$
 In other words $\cC$, $\cC_{\oplus}$ and $\operatorname{Idem}(\cC_{\oplus})$  are Morita equivalents. In particular, we obtain the following result. 

\begin{prop}\label{eq}
Let $\cC$ be a $\Z$-linear category. The following are equivalent:

\begin{enumerate}
\item[(1)] $\cC$ is  right regular $n$-coherent.
\item[(2)] $\cC_{\oplus}$ is right regular $n$-coherent.
\item[(3)] $\operatorname{Idem}(\cC_{\oplus})$ is right regular $n$-coherent. 
\end{enumerate}
\end{prop}

Let $R$ be a ring with unity. A finitely $n$-presented right $R$-module $M$ is $n$-coherent if every finitely $(n-1)$-presented submodule $N\subseteq M$ is finitely $n$-presented. The ring $R$ is right $n$-coherent if $R$ is $n$-coherent as a right $R$-module (i.e. if each finitely $(n-1)$-presented right ideal of $R$ is finitely $n$-presented). We say that $R$ is strong right $n$-coherent if each finitely $n$-presented right $R$-module is finitely $(n+1)$-presented. A strong $n$-coherence ring is equivalent to a $n$-coherence ring for $n =1$, but it is an open question for $n \geq 2$. 
A coherent ring is a 1-coherent ring (strong 1-coherent ring) and it is regular if and only if every finitely presented module has finite projective dimension. Motivated by this we introduce in \cite[Definition 2.9]{ep} the definition of $n$-regular ring. 
Let $n \geq 1$, a ring $R$ is called  right \emph{$n$-regular} if each finitely $n$-presented  right $R$-module has finite projective dimension.

\begin{cor} Let $R$ be a ring with unity and $n\geq 1$. Then the following are equivalent. 

\begin{enumerate}
\item[(1)] The ring $R$ is  strong right $n$-coherent or right $n$-regular and  strong right $n$-coherent respectively;
\item[(2)] The additive category $\underline{R}_{\oplus}$ is right $n$-coherent or right regular $n$-coherent respectively;
\item[(3)] The additive category $Idem(\underline{R}_{\oplus})$ is  right $n$-coherent or right regular $n$-coherent respectively.
\end{enumerate}
\end{cor}

Let $\cC$ be a small additive category. By \cite[Lemma 6.8]{bl}, $\cC$ is right Noetherian if and only if each object $c$ has the following property. Consider any directed set $I$ and collections of morphisms $\lbrace{f_i: a_i\rightarrow c}\rbrace_{i\in I}$ with $c$ as target such that $f_i\subseteq f_j$ holds for $i\leq j$, then there exists $i_0 \in I$ with $f_i\subseteq f_{i_0}$
for all $i\in I$.  Our aim is to find out intrinsic condition of $\cC$ which guarantees that $\operatorname{Fun}(\cC^{op}, \operatorname{Ab})$ is regular $n$-coherent.

\begin{prop}\label{char} Let $\cC$ be a small additive category and $n\geq 1$. The following are equivalent 
\begin{enumerate}
\item[(1)] $\cC$ is  right regular $n$-coherent.
\item[(2)] The following conditions hold in $\cC$:
\begin{enumerate}
\item[i)] Every morphism in $\cC$ with a pseudo $(n-1)$-kernel has a pseudo $n$-kernel. 
\item[ii)] For every morphism $f:x\rightarrow y$ in $\cC$ with pseudo $\infty$-kernel there exists $k\in\mathbb{N}$ and $\alpha:x_{k-1}\rightarrow x_{k-1}$ making the following diagram commute:

$$\xymatrix{
x_k\ar[rd]^{0} \ar[r]^{f_{k}} & x_{k-1}\ar[d]^{\alpha} \ar[r]^{f_{k-1}} & x_{k-2} \ar[r]^{f_{k-2}}& \cdots
&\cdots \ar[r]^{f_{2}}& x_{1} \ar[r]^{f_{1}}& x \ar[r]^{f}& y \\
& x_{k-1}\ar[ru]_{f_{k-1}}}$$

\end{enumerate}
\end{enumerate}
\end{prop}

\begin{proof} 
 $(1\Rightarrow 2)$ Suppose that $\cC$ is right regular $n$-coherent.  First, we note that $(i)$ is clear by \cite[Prop 5.4]{BOPM}. Now suppose that $f:x\rightarrow y$ is a morphism in $\cC$ with a pseudo $\infty$-kernel $(\cdots,f_{3},f_{2},f_{1})$. Here, we let $f_0:=f$. Thus $\coker(f_*)$ is of type $\mathcal{FP}_{\infty}$ in $\operatorname{Fun}(\cC^{op}, \operatorname{Ab})$ because there exists an exact sequence of the form
$$\cdots \xrightarrow{f_{2*}} \hom_\cC(-,x_1)\xrightarrow{f_{1*}} \hom_\cC(-,x)\xrightarrow{f_*} \hom_\cC(-,y)\rightarrow \coker(f_*)\rightarrow 0.$$ 
 There exists $k\in\mathbb{N}$ such that $\coker(f_*)$ has  projective dimension $\leq k.$  It implies that $\ker(f_{k-2*})=\im(f_{k-1*})$ is projective, and therefore $\ker(f_{k-1*})=\im(f_{k*})$ is projective too. Consider
 $$\xymatrix{
\cdots\rightarrow \hom_\cC(-,x_{k})\ar@{->>}[rd]^{\sigma} \ar[r]^{f_{k*}}
& \hom_\cC(-,x_{k-1}) \rightarrow\cdots
&  \\
&\im(f_{k*})\ar@{^{(}->}[u]^{\iota}}$$
where $\iota:\im(f_{k*})\hookrightarrow \hom_\cC(-,x_{k-1})$ and $\sigma:\hom_\cC(-,x_{k})\twoheadrightarrow \im(f_{k*})$ are the canonical morphisms. 
There exists $\iota':\im(f_{k*})\rightarrow \hom_\cC(-,x_{k})$  and $\sigma': \hom_\cC(-,x_{k-1})\rightarrow \im(f_{k*})$  such that $\sigma\circ\iota'= \id_{\im(f_{k*})}$ and $\sigma'\circ\iota=\id_{\im(f_{k*})}$. By Yoneda Lemma and using the same techniques \cite[Theorem 5.5]{BOPM} there exists $h:x_{k-1}\rightarrow x_{k}$ in $\cC$ such that $h_*:\hom_\cC(-,x_{k-1})\rightarrow \hom_\cC(-,x_{k})$ satisfy $h_*\circ f_{k*} = \iota'\circ\sigma.$ The morphism $\alpha:= \id_{x_{k-1}}- f_{k}\circ h$ satisfies the desired condition. 

$(2\Rightarrow 1)$ Suppose that the affirmation $(2)$ is satisfied for $n\geq1$. Using the condition $(2$-$i)$, we deduce that $\cC$ is right $n$-coherent \cite[Prop 5.4]{BOPM} and thus $\mathcal{FP}_{n}=\mathcal{FP}_{\infty}.$ Now, for each $F: \cC^{op} \rightarrow{Ab}$ of type $\mathcal{FP}_{n}$  we get an exact sequence of the form
$$ \cdots \rightarrow\hom_\cC(-, x_{n})\rightarrow \cdots\rightarrow \hom_\cC( -,x_1)\xrightarrow{f_{1*}} \hom_\cC( -,x)\xrightarrow{f_*} \hom_\cC( -,y)\rightarrow F\rightarrow 0$$
where $f:x\rightarrow y$ is a morphism in $\cC.$ It implies that $f$ has a pseudo $\infty$-kernel, and therefore, there is  $k\in\mathbb{N}$ and an endomorphism $\alpha:x_{k-1}\rightarrow x_{k-1}$ making the following diagram commute:
$$\xymatrix{
x_k\ar[rd]^{0} \ar[r]^{f_{k}} & x_{k-1}\ar[d]^{\alpha} \ar[r]^{f_{k-1}} & x_{k-2}\\
& x_{k-1}\ar[ru]_{f_{k-1}}}$$
Next, we show that $\im(f_{k-1*})=\ker(f_{k-2*})$ is a projective functor. Consider
$$
\xymatrix{
\hom_{\cC}(-, x_{k}) \ar[ddr]_{0}\ar[r]^{f_{k*}}& \hom_{\cC}(-, x_{k-1})\ar[dr]^{\sigma}\ar[dd]_{\alpha_*} \ar[rr]^{f_{k-1*}}&& \hom_{\cC}(-, x_{k-2}) \ar@{=}[dd]\\
&&\im(f_{k-1*})\ar[ur]^{\iota}&\\
& \hom_{\cC}(-, x_{k-1})\ar[rr]^{f_{k-1*}}&& \hom_{\cC}(-, x_{k-2})
}
$$
where $\sigma:\hom_{\cC}(-,x_{k-1})\rightarrow \im(f_{k-1*})$ and $\iota:\im(f_{k-1*})\rightarrow \hom_{\cC}(-,x_{k-2})$ are the canonical natural transformations.
Note that $\im(f_{k-1*})=\coker(f_{k*})$, then there exists unique natural transformation $t:\im(f_{k-1*})\rightarrow \hom_{\cC}(-,x_{k-1})$ such that $t\circ\sigma=\alpha_*.$ Moreover, applying the same techniques \cite[Theorem 5.5]{BOPM} we have
$$\iota\circ\id_{\im(f_{k-1*})}\circ\sigma=\iota\circ\sigma=f_{k-1*}=(f_{k-1}\circ\alpha)_*=f_{k-1*}\circ\alpha_*= f_{k-1*}\circ t\circ\sigma=\iota\circ\sigma\circ t\circ\sigma$$
which implies that $$\id_{\im(f_{k-1*})}=\sigma\circ\ t.$$ Then $\sigma$ is a split epimorphism, and therefore, $\im(f_{k-1*})$ is projective.
\end{proof}
According to \cite{bl}, a small additive category $\cC$ is considered to be right regular if it satisfies two conditions: it is both right Noetherian and right regular 1-coherent. It´s important to note that this usage of regular should not be confused with  the concept of von Neumann regular. 

\begin{cor} Let $\cC$ be a small additive category. The following are equivalent
\begin{enumerate}
\item[(1)] $\cC$ is  right regular.
\item[(2)] The following conditions hold in $\cC$:
\begin{enumerate}
\item[i)] Every object $c$ in $\cC$ has the following property. Consider any directed set $I$ and collections of morphisms $\lbrace{f_i: a_i\rightarrow c}\rbrace_{i\in I}$ with $c$ as target such that $f_i\subseteq f_j$ holds for $i\leq j$. Then there exists $i_0 \in I$ with $f_i\subseteq f_{i_0}$ for all $i\in I$.
\item[ii)]For every morphism $f:x\rightarrow y$ in $\cC$ with pseudo $\infty$-kernel there exists $k\in\mathbb{N}$ and  $\alpha:x_{k-1}\rightarrow x_{k-1}$ making the following diagram commute:

$$\xymatrix{
x_k\ar[rd]^{0} \ar[r]^{f_{k}} & x_{k-1}\ar[d]^{\alpha} \ar[r]^{f_{k-1}} & x_{k-2} \ar[r]^{f_{k-2}}& \cdots
&\cdots \ar[r]^{f_{2}}& x_{1} \ar[r]^{f_{1}}& x \ar[r]^{f}& y \\
& x_{k-1}\ar[ru]_{f_{k-1}}}$$

\end{enumerate}
\end{enumerate}
\end{cor}

In \cite{bl} another type of regularity is introduced due to bad behavior of regularity with respect to infinity products. 
Let $R$ be a ring with unity. Specifically, $R$ is right \emph{$l$-uniformly regular coherent}, if every finitely presented right $R$-module $M$ admits a $l$-dimensional finite projective resolution, i.e. there exists an exact sequence  
$$0\rightarrow P_{l}\rightarrow P_{l-1}\rightarrow\cdots\rightarrow P_0\rightarrow M\rightarrow 0$$ where each $P_i$ is finitely generated and projective right $R$-module. This concept is extended to additive categories in \cite[Section 6]{bl}. 
Let $\cC$ be a $\Z$-linear category  and $l\geq 1$. We say that $\cC$ is right \emph{$l$-uniformly regular coherent}, if every right  $\cC$-module $F$ of type $\mathcal{FP}_{1}$  admits a $l$-dimensional finite projective resolution, i.e. there exists an exact sequence 
$$0\rightarrow P_{l}\rightarrow P_{l-1}\rightarrow\cdots\rightarrow P_0\rightarrow F\rightarrow 0$$ where each $P_i$ is finitely generated and projective right $\cC$-module. 

The equivalence  $\operatorname{Fun}(\cC^{op}, \operatorname{Ab})\simeq \operatorname{Fun}(\cC^{op}_{\oplus}, \operatorname{Ab})$ implies that $\cC$ is right $l$-uniformly regular coherent if and only if $\cC_{\oplus}$  is right $l$-uniformly regular coherent. Note that, if $\cC$ is right $1$-coherent and every right $\cC$-module $F$ of type $\mathcal{FP}_{1}$ has a projective dimension $\leq l$ then $\cC$ is right $l$-uniformly regular coherent. 

\begin{cor} Let $l\geq 1$ and let $\cC$ be a small additive category. Suppose that $\cC$ is right $1$-coherent. Then, the following are equivalent:
\begin{enumerate}
\item[(1)] $\cC$ is right $l$-uniformly regular coherent.
\item[(2)] For every morphism $f:x\rightarrow y$ in $\cC$ there exists $l\in\mathbb{N}$, an pseudo $l$-kernel $(f_l,f_{l-1},\cdots ,f_1)$ of $f$ and $\alpha:x_{l-1}\rightarrow x_{l-1}$ making the following diagram commute:

$$\xymatrix{
x_l\ar[rd]^{0} \ar[r]^{f_{l}} & x_{l-1}\ar[d]^{\alpha} \ar[r]^{f_{l-1}} & x_{l-2} \ar[r]^{f_{l-2}}& \cdots
&\cdots \ar[r]^{f_{2}}& x_{1} \ar[r]^{f_{1}}& x \ar[r]^{f}& y \\
& x_{l-1}\ar[ru]_{f_{l-1}}}$$

\end{enumerate}
\end{cor}

\section{The ring $\cA(\cC)$ and the $\Z$-linear category $\cC$}\label{sec3}
In this section we study the relation between some properties of a $\Z$-linear category $\cC$ with the properties of a ring $\cA(\cC)$ associated with it. We prove the categories $\operatorname{Fun}(\cC^{op},\operatorname{Ab})$ and $\operatorname{Mod}$-$\cA(\cC)$
are equivalent.

 \subsection{The ring $\cA(\cC)$}
Let $\cC$ be a $\Z$-linear category. Recall from \cite{ce}
\begin{equation}\label{ac}
\cA(\cC)=\bigoplus_{a,b\in\obc}\hom_\cC(a,b).
\end{equation}
If $f\in \cA(\cC)$ write $f_{a,b}$ for the component in
$\hom_\cC(b,a)$. The following multiplication law
\begin{equation}\label{rule:matrix}
(fg)_{a,b}=\sum_{c\in\obc}f_{a,c}g_{c,b}
\end{equation}
makes $\cA(\cC)$ into an associative ring, which is unital if and only if $\obc$ is finite. Whatever the cardinal
of $\obc$ is, $\cA(\cC)$ is always a ring with {\it local units}, i.e. a filtering colimit of unital rings.

\subsection{The $\Z\cC$-modules}

Recall that $M$ is a unital right $\cA(\cC)$-module if $M\cdot \cA(\cC)=M$. Consider 
$\operatorname{Mod}$-$\cA(\cC)$ the category of unital right $\cA(\cC)$-modules. Let us define functors
$$
\cS(-): \operatorname{Fun}(\cC^{op}, \operatorname{Ab}) \rightarrow \mbox{$\operatorname{Mod}$-$\cA(\cC)$} \qquad
(-)_{\cC}: \mbox{$\operatorname{Mod}$-$\cA(\cC)$}  \rightarrow \operatorname{Fun}(\cC^{op}, \operatorname{Ab})$$
Let $M \in \operatorname{Fun}(\cC^{op}, \operatorname{Ab})$
$$
\cS(M)=\bigoplus _{a\in \obc}M(a) 
$$
Let $N\in \mbox{$\operatorname{Mod}$-$\cA(\cC)$}$
$$
N_{\cC}:\cC^{op} \rightarrow  \operatorname{Ab} \qquad a \mapsto N\cdot \id_{a}.
$$

\begin{lem}
If $N$ is a unital right $\cA(\cC)$-module then $$\bigoplus_{a\in \obc} N\cdot \id_{a} = N.$$
\end{lem}
\begin{proof}
For every $a \in \obc$ we have $N\cdot \id_{a} \subseteq N$ then 
$\bigoplus_{a\in \obc}N\cdot \id_{a} \subseteq N$.
Let $n\in N$, because $N$ is unital $N=N\cdot \cA(\cC)$ then $n=\sum^{i=m}_{i=1} n_{i}\cdot f_{i}$ with $n_{i}\in  N$ and $f_{i}\in \hom_{\cC}(a_{i},b_{i})$.
Let $I=\{a  \in \obc: a=a_{i}, \mbox{for some $i=1,\ldots, m$}\}$ then
$$
n= \sum^{i=m}_{i=1} n_{i}\cdot f_{i} = (\sum^{i=m}_{i=1} n_{i}\cdot f_{i})\cdot(\sum_{a\in I} \id_{a})= n\cdot \sum_{a\in I} \id_{a}
$$
We conclude $N \subseteq  \bigoplus_{a\in \obc}N\cdot \id_{a}$.
\end{proof}
\begin{prop}\label{propeq}
Let $\cC$ be a $\Z$-linear category then $$
\cS(-): \operatorname{Fun}(\cC^{op}, \operatorname{Ab}) \rightarrow \mbox{$\operatorname{Mod}$-$\cA(\cC)$} \qquad
(-)_{\cC}: \mbox{$\operatorname{Mod}$-$\cA(\cC)$}  \rightarrow \operatorname{Fun}(\cC^{op}, \operatorname{Ab})$$
are an equivalence of categories.  
\end{prop}
\begin{proof}
Let $N\in \mbox{$\operatorname{Mod}$-$\cA(\cC)$}$ and  $M\in \operatorname{Fun}(\cC^{op}, \operatorname{Ab}) $ then 
$$
S(N_{\cS})= \bigoplus_{a\in \obc} N_{\cC}(a) = \bigoplus_{a\in \obc} N\cdot \id_{a} = N
$$
$$
(S(M))_{\cC}(c)=S(M)\cdot \id_{c}=\bigoplus _{a \in \obc} M(a)\cdot \id_{c}=M(c) \qquad \forall c\in \obc
.$$
\end{proof}

The abelian structure of $\operatorname{Fun}(\cC^{op}, \operatorname{Ab})$ comes from the abelian structure in $\operatorname{Ab}$. A sequence $M\xrightarrow{f} N \xrightarrow{g} R$ is exact in $\operatorname{Fun}(\cC^{op}, \operatorname{Ab})$ if for each object $c\in \cC$ the sequence $M(c)\xrightarrow{f(c)} N(c) \xrightarrow{g(c)} R(c)$  is exact in  $\operatorname{Ab}$.

\begin{prop}\label{propsc}
Let $\cC$ be a $\Z$-linear category then $$
\cS(-): \operatorname{Fun}(\cC^{op}, \operatorname{Ab}) \rightarrow \mbox{$\operatorname{Mod}$-$\cA(\cC)$} \qquad
(-)_{\cC}: \mbox{$\operatorname{Mod}$-$\cA(\cC)$}  \rightarrow \operatorname{Fun}(\cC^{op}, \operatorname{Ab})$$
are exact functors. 
\end{prop}
\begin{proof}
Let $M\xrightarrow{f} N \xrightarrow{g} R$ be an exact sequence in $\operatorname{Mod}$-$\cA(\cC)$. Let us prove $M_{\cC}\xrightarrow{f_{\cC}} N_{\cC} \xrightarrow{g_{\cC}} R_{\cC}$  is exact in $\operatorname{Fun}(\cC^{op}, \operatorname{Ab})$ showing $M_{\cC}(a)\xrightarrow{f_{\cC}(a)} N_{\cC}(a) \xrightarrow{g_{\cC}(a)} R_{\cC}(a)$ is exact for every object $a$ in $\cC$.
By functoriality $\im(f_{\cC}(a)) \subseteq \ker (g_{\cC}(a))$. Let $n\cdot \id_{a} \in \ker (g_{\cC}(a))$ then 
$$g_{\cC}(a)(n\cdot \id_{a})=g(n)\cdot \id_{a}=g(n\cdot \id_{a})=0$$
then $n\cdot \id_{a} \in \ker(g)=\im (f)$. There exists $m\in M$ such that $f(m)=n\cdot \id_{a}$ then
$$
f_{\cC}(a)(m\cdot \id_{a})= f(m\cdot \id_{a})=f(m)\cdot \id_{a} = (n\cdot \id_{a})\cdot \id_{a}= n \cdot \id_{a}
$$
then  $n \cdot \id_{a}\in \im(f_{\cC}(a))$. We conclude $(-)_{\cC}$ is exact. 

We proceed to show that $S$ is exact. Let $M\xrightarrow{f} N \xrightarrow{g} R$ be an exact sequence in $\operatorname{Fun}(\cC^{op}, \operatorname{Ab})$. Consider 
 $$
 \displaystyle S(M) = \bigoplus_{a\in \obc} M(a)\xrightarrow{S(f)} S(N) = \bigoplus_{a\in \obc} N(a) \xrightarrow{S(g)} S(R)=\bigoplus_{a\in \obc} R(a)
 $$
 Similarly as above, let $\sum _{a\in\cC} x_{a} \in \ker S(g)$ then 
 $$
 \begin{array}{lll}
 S(g)(\sum _{a\in\cC} x_{a}) =\sum_{a\in\cC}g(a)(x_{a})=0 & \Rightarrow & g(a)(x_{a})=0 \qquad \forall x_{a}\in N(a) \\
 & \Rightarrow & x_{a}\in \ker g(a)=\im f(a) \qquad \forall x_{a}\in N(a)\\
&\Rightarrow  & \exists y_{a} \in M(a) \mbox{ such that } f(a)(y_{a})=x_{a}
 \end{array}
 $$
 \end{proof}

\begin{cor}\label{propfunc}
Let $\cC$ be a $\Z$-linear category.
\begin{enumerate}
\item If $p: M \rightarrow N$ is an epimorphism in $\operatorname{Mod}$-$\cA(\cC)$ then $p_{\cC}:M_{\cC} \rightarrow N_{\cC}$ is an epimorphism in $\operatorname{Fun}(\cC^{op}, \operatorname{Ab})$. 
\item  If $\pi: M \rightarrow N$ is an epimorphism in $\operatorname{Fun}(\cC^{op}, \operatorname{Ab})$ then $\cS(\pi) : \cS(M)  \rightarrow \cS(N)$ is an epimorphism in $\operatorname{Mod}$-$\cA(\cC)$. 
\item $(M \oplus N)_{\cC} = M_{\cC}\oplus N_{\cC}$ in $\operatorname{Fun}(\cC^{op}, \operatorname{Ab})$.
\item $\cS(M \oplus N) = \cS(M) \oplus \cS(N)$ in $\operatorname{Mod}$-$\cA(\cC)$.
\end{enumerate}
\end{cor}

Let $A$ be a ring with local units. From \cite{wis} we recall that an $A$-module is {\it quasi-free}  if it is isomorphic to a direct sum of modules of the form $e\cdot A$ with $e^{2}=e$, $e\in A$. 
Quasi-free modules over a ring with local units play the same role as free modules over a ring with unity.
Also recall that $M$ is a finitely generated module if and only if it is an image of a finitely generated quasi-free module. A finitely generated module $M$ is projective  if and only if it is a direct summand of a finitely generated quasi-free modules. 
In this paper we work with $A=\cA(\cC)$ and we say that $M$ is a quasi-free right $\cA(\cC)$-module if it is isomorphic to a finite sum of modules $\id_{a}\cdot \cA(\cC)$.

\begin{lem}\label{lempl}Let $\cC$ be a $\Z$-linear category.
\begin{enumerate}
\item If $F$ is a free finitely generated right $\cC$-module, then $S(F)$ is a quasi-free finitely generated right $\cA(\cC)$-module.
\item If $P$ is a projective finitely generated right $\cC$-module, then $S(P)$ is projective and  finitely generated right $\cA(\cC)$-module.
\item If $M$ is a quasi-free finitely generated right $\cA(\cC)$-module then $M_{\cC}$ is a finitely generated free right $\cC$ module.
\item \label{cinco} If $P$ is a projective finitely generated right $\cA(\cC)$-module then  $P_{\cC}$ is projective and finitely generated right $\cC$-module.
\end{enumerate}
\end{lem}
\begin{proof}
\begin{enumerate}
\item Let $I$ be a finite subset of objects in $\cC$ such that $F = \bigoplus_{b\in I }\hom_{\cC}(-,b)$. Then we have
$$S(F)= \bigoplus_{b\in I }S(\hom_{\cC}(-,b))=\bigoplus_{b\in I }\id_{b} \cdot \cA(\cC) $$
This shows that $S(F)$ is a quasi-free finitely generated  right $\cA(\cC)$-module.
\item  Suppose $P$ is a finitely generated projective right $\cC$-module. Then there exists a module $Q$ such that $P\oplus Q = F$, where $F$ is a free module. Moreover, we have $S(P)\oplus S(Q)=S(F)$, where $S(F)$ is quasi-free and finitely generated. Therefore, $S(P)$ is also projective. 
\item Suppose $M$ is a quasi-free finitely generated right $\cA(\cC)$-module. Then there exists a finite set $I$ such that $M=\bigoplus_{b\in I}\id_{b}\cdot \cA(\cC)$. Note that for any object $a$ in $\cC$,
$$
M_{\cC}(a)=M\cdot \id_{a} =(\bigoplus_{b\in I}\id_{b}\cdot \cA(\cC)) \cdot  \id_{a} = \bigoplus_{b\in I}\hom_{\cC}(a,b)$$
Therefore, we have
$$
M_{\cC}=\bigoplus_{b\in I}\hom_{\cC}(-,b)
$$
which is a free finitely generated module in  $\operatorname{Fun}(\cC^{op}, \operatorname{Ab})$. 
\item  Suppose $P$ is a projective finitely generated $\cA(\cC)$-module. Then there exists a module $Q$ such that $P\oplus Q = F$, where $F$ is a quasi-free finitely generated $\cA(\cC)$-module. We have
 $$P_{\cC}\oplus Q_{\cC} = F_{\cC}.$$
 Therefore, $P_{\cC}$ is also projective and finitely generated.
\end{enumerate}
\end{proof}

\begin{prop}\label{catring}
Let $\cC$ be a $\Z$-linear category with finitely many objects and $n\geq 1$.
\begin{enumerate}
\item The category $\cC$ is right Noetherian  if and only if $\cA(\cC)$  is a right Noetherian ring.
\item The category $\cC$ is right $n$-coherent if and only if $\cA(\cC)$ is a strong right  $n$-coherent ring. 
\item The category $\cC$ is regular $n$-coherent if and only if  $\cA(\cC)$ is a right $n$-regular and strong right $n$-coherent ring. 
\end{enumerate}
\end{prop}
\begin{proof}
\begin{enumerate}
\item  Let $M$ be a finitely generated right $\cA(\cC)$-module and let $N$ be a submodule.  Consider the epimorphism
$$
\cA(\cC) \oplus \ldots \oplus \cA(\cC) \rightarrow M, 
$$
and let us apply Corollary \ref{propfunc} to obtain the following epimorphism:
$$
\cA(\cC)_{\cC} \oplus \ldots \oplus \cA(\cC)_{\cC} \rightarrow M_{\cC} .
$$
As $\cA(\cC)_{\cC} =\bigoplus_{b\in\obc}\hom_\cC(-,b)$ we obtain that $M_{\cC}$ is finitely generated. 
$$
\cA(\cC)_{\cC} \oplus \ldots \oplus \cA(\cC)_{\cC}  = \displaystyle\bigoplus_{j \in J \ b_{j}\in \obc } \hom_{\cC} (-,b_{j})
$$
Since $\cC$ is right Noetherian, we can conclude that $N_\cC$ is also finitely generated.  Moreover, there exists an epimorphism
$$
 \displaystyle\bigoplus_{i \in I} \hom_{\cC} (-,a_{i})\rightarrow N_{\cC}
$$ then 
$$
 \displaystyle\bigoplus_{i \in I} \cS (\hom_{\cC} (-,a_{i}))\rightarrow \cS (N_{\cC})=N.
$$ 
Consider the projection 
$$
p_{i}:  \cA (\cC) \rightarrow \cS (\hom_{\cC} (-,a_{i}))=\bigoplus_{c\in \obc}\hom_{\cC} (c,a_{i})
$$
Taking $n=\#I$ we obtain an epimorphism
$$
\cA (\cC)^{n} \rightarrow  \displaystyle\bigoplus_{i \in I} \cS (\hom_{\cC} (-,a_{i}))\rightarrow N,
$$ 
then $N$ is finitely generated. 

Conversely if $M\in \operatorname{Fun}(\cC^{op}, \operatorname{Ab})$ is finitely generated let us show that every subobject is also finitely generated.  Take $N$ as a submodule of $M$. 
There is an epimorphism
$$
 \displaystyle\bigoplus_{i \in I} \hom_{\cC} (-,a_{i})\rightarrow M 
$$
then we have an epimorphism 
$$
 \displaystyle\bigoplus_{i \in I, c\in\obc} \hom_{\cC} (c,a_{i})= \displaystyle\bigoplus_{i \in I} S(\hom_{\cC} (-,a_{i}))\rightarrow S(M). 
$$
We obtain that $S(N)$ is a submodule of $S(M)$ which is finitely generated, then $S(N)$ is also finitely generated and $S(N)_{\cC} =N$ is finitely generated. 
\item Let $M$ be a finitely $n$-presented right $\cA(\cC)$-module. Consider $m_{0}$, $m_{1}$,  $\ldots$, $m_{n}$ $\in \N$ such that 
$$\cA(\cC)^{m_{n}}\rightarrow \cA(\cC)^{m_{n-1}} \rightarrow \ldots \rightarrow \cA(\cC)^{m_{1}} \rightarrow \cA(\cC)^{m_{0}} \rightarrow M \rightarrow 0 $$ 
is exact. By Proposition \ref{propsc} the following is also an exact sequence
$$\cA(\cC)_{\cC}^{m_{n}}\rightarrow \cA(\cC)_{\cC}^{m_{n-1}} \rightarrow \ldots \rightarrow \cA(\cC)_{\cC}^{m_{1}} \rightarrow \cA(\cC)_{\cC}^{m_{0}} \rightarrow M_{\cC} \rightarrow 0 $$ 
As $\cA(\cC)_{\cC} =\bigoplus_{b\in\obc}\hom_\cC(-,b)$ we obtain that $M_{\cC}$ is of type $\mathcal{FP}_{n}$. Because $\cC$ is right $n$-coherent there exists an exact sequence
$$
\cdots \rightarrow  P_{n+1}\rightarrow  P_{n} \rightarrow \ldots \rightarrow P_{1}\rightarrow P_{0} \rightarrow M_{\cC} \rightarrow 0
$$
where each $P_{i}$ is both projective and finitely generated. Then, 
$$
\cdots \rightarrow  S(P_{n+1})\rightarrow S( P_{n}) \rightarrow \ldots \rightarrow S(P_{1})\rightarrow S(P_{0}) \rightarrow M \rightarrow 0
$$
is exact and by Lemma \ref{lempl} $S(P_{i})$ is projective and finitely generated. Therefore,  $\cA(\cC)$ is a strong right $n$-coherent ring. 

Conversely, if $F\in \operatorname{Fun}(\cC^{op}, \operatorname{Ab})$ is of type $\mathcal{FP}_{n}$ then $S(F)$ is an $\cA(\cC)$-module of the type $\mathcal{FP}_{n}$. 
As $\cA(\cC)$ is a strong  right $n$-coherent ring there exists $P_{i}$ projective finitely generated $\cA(\cC)$-modules such that
$$\ldots \rightarrow P_{n} \rightarrow \cdots \rightarrow P_{1} \rightarrow P_{0} \xrightarrow \ S(F) \rightarrow 0$$
Then 
$$ \ldots \rightarrow (P_{n})_{\cC} \rightarrow \cdots \rightarrow (P_{1})_{\cC} \rightarrow (P_{0})_{\cC} \rightarrow F \rightarrow 0$$
where $(P_{i})_{\cC}$ are projective and finitely generated by Lemma \ref{lempl}\eqref{cinco}.
\item Let $M$ be a finitely $n$-presented right $\cA(\cC)$-module. By the previous item, we know that   $M_{\cC}$ is of type $\mathcal{FP}_{n}$. 

Since $\cC$ is right $n$-regular, there exists an exact sequence
$$
0 \rightarrow  P_{k}\rightarrow  P_{k-1} \rightarrow \ldots \rightarrow P_{1}\rightarrow P_{0} \rightarrow M_{\cC} \rightarrow 0
$$
where each $P_i$ is a finitely generated projective module. Then, the sequence
$$
0 \rightarrow  S(P_{k})\rightarrow S( P_{k-1}) \rightarrow \ldots \rightarrow S(P_{1})\rightarrow S(P_{0}) \rightarrow M \rightarrow 0
$$

is exact, and by Lemma \ref{lempl}, $S(P_{i})$ is also finitely generated and projective. Therefore, $\cA(\cC)$ is a right $n$-regular and strong right $n$-coherent ring. The conversely is similar. 
\end{enumerate}
\end{proof}

\begin{ex} Let us consider some examples of $\Z$-linear categories with finitely many  objects. 

\begin{enumerate}
    \item Let $R$ be a ring and $G=\mathbb{Z}_{n}$. Consider $\tilde{R}=\frac{R[t]}{<t^{n}>}$. The category $\cC_{\tilde{R}}$ is the category with $n$ objects and 
$$
\operatorname{hom}_{\cC_{\tilde{R}}}(p,q)=\tilde{R}_{q-p}=R
$$
Note $\cA(\cC_{\tilde{R}})=M_{n\times n}(R)$. 
If $R$ is a Noetherian ring, then $\cA(\cC_{\tilde{R}})$ is also Noetherian.
By Proposition \ref{catring}, then $\cC_{\tilde{R}}$ is Noetherian. 

\item We recall from \cite{Costa} that a ring $R$ is said to be $(n,d)$-ring if every $n$-presented $R$-module has projective dimension at most $d$.
Remark that if $n\leq n'$ and $d\leq d'$, then every $(n,d)$-ring is also a $(n',d')$-ring.  

Let $R$, $S$ be a finite direct sum of fields and $\cC$ be the $\Z$-linear category with two objects $a$ and $b$ such that 
$\hom_{\cC}(a,b)=\hom_{\cC}(b,a)=0$, $\hom_{\cC}(a,a)=R$ and $\hom_{\cC}(b,b)=S$. Notice $\cA(\cC)=R\oplus S$, by 
\cite[Theorem 1.3 (i)]{Costa} $\cA(\cC)$ is a $(0,0)$-ring and hence a Noetherian and regular coherent ring. 

\item Let $G$ be a finite commutative group. An associative ring $R$ graded by $G$ is
$$R=\bigoplus_{g\in G} R_g$$
such that the multiplication satisfies $R_g · R_h \subseteq R_{g+h}$ for all $g,h \in G$. A (left) graded module over $R$ is an $R$-module $M$ together with a decomposition $M=\bigoplus_{g\in G} M_g$ such that $R_g M_h \subseteq M_{g+h}$. We denote by $R$-GrMod the category of graded $R$-modules. The category $\cC_R$ is the $\Z$-linear category whose set of objects is $\lbrace{g: g\in G}\rbrace$ and whose morphism groups are given by  $\hom_{\cC_R}(g,h)= R_{h-g}$. By \cite[Lemma 2.2]{DS} there is an equivalence between $R$-GrMod and the additive functor category $\operatorname{Fun}(\cC_R, \operatorname{Ab})$.

\end{enumerate}
\end{ex}

\section{K-theory of $\mathbb{Z}$-linear categories}\label{sec4}

\subsection{Vanishing negative K-theory} In this section, we have a result of vanishing negative K-theory of $\Z$-linear categories. 
Recall from \cite[Section 4]{ce} the definition of the K-theory spectrum of a $\mathbb{Z}$-linear category $\cC$, the K-theory spectrum of the ring $\cA(\cC)$ and the map 
\begin{equation}\label{nemap}
\varphi: K(\mathcal{C}) \rightarrow K(\cA(\cC))
\end{equation}
which is a natural equivalence in $\cC$, see \cite[Proposition 4.2.8]{ce}.

\begin{thm}\label{teofo}
Let $\cC$ be a $\Z$-linear category with finitely many objects. 
\begin{enumerate}
\item If $\cC$ is right regular, then $K_{i}(\cC)=0$  for all $i<0.$
\item If $\cC$ is right regular coherent, then $K_{-1}(\cC)=0$.
\end{enumerate}
\end{thm}

\begin{proof}
Assume that $\cC$ is a right regular category. Then, by Proposition \ref{catring}, $\cA(\cC)$ is a right regular ring. By the fundamental theorem of K-theory, we have $K_i(\cA(\cC)) = 0$ for all $i < 0$. It follows that
 $$ K_{i} (\cC) \simeq K_{i}(\cA(\cC))=0 \qquad \forall i<0.$$

Now, assume that $\cC$ is a right regular coherent category. Then $\cA(\cC)$ is a right regular coherent ring, by Proposition \ref{catring}. By \cite[Theorem 3.30]{ant}, we have $K_{-1}(\cA(\cC)) = 0$, and $$ K_{-1} (\cC) \simeq K_{-1}(\cA(\cC))=0.$$
\end{proof}

 \begin{cor} Let $\mathcal{D}=\cC_{\oplus}$ with $\cC$ be a $\Z$-linear category with finitely many objects. 
 \begin{enumerate}
\item If $\cC$ is right regular, then $K_{i}(\mathcal{D})=0$ for all $i<0$.
\item If $\cC$ is right regular coherent, then $K_{-1}(\mathcal{D})=0$.
\end{enumerate}
\end{cor}
 
\begin{defn} A $\Z$-linear category $\cC$ is right {\emph{AF-regular}} if there is $\{ \cC_{f}\}_{f\in F}$ a direct system of right regular $\Z$-linear categories with finitely many  objects such that 
$$
\cC = \colim_{f\in F} \cC_{f} 
$$
Similarly, we say that $\cC$ is right {\emph{AF-Noetherian}} ({\emph{AF-regular coherent})} if 
$$
\cC = \colim_{f\in F} \cC_{f} 
$$
with $\cC_{i}$ being directed systems of right Noetherian (regular coherent) $\Z$-linear categories with finitely many objects. 
\end{defn}

\begin{thm}\label{teoaf}
Let $\cC$ be a $\Z$-linear category. 
\begin{enumerate}
\item If $\cC$ is right AF-regular, then $K_{i}(\cC)=0$ $\forall i<0$.
\item If $\cC$ is right AF-regular coherent, then $K_{-1}(\cC)=0$.
\end{enumerate}
\end{thm}
\begin{proof}
Assume that  $\cC = \colim_{f\in F} \cC_{f}$. Using the continuity of $K$-theory, we have $K_{i}(\cC)=  \colim_{f\in F}K_{i} (\cC_{f})$ for all $i<0$. The rest of the proof follows from Theorem \ref{teofo}. 
\end{proof}

\subsection{Fundamental Theorem and homotopy invariance}
Let $\cC$ be a $\Z$-linear category. We consider the category $\cC[t]$ with the same objects of $\cC$ and morphisms are
$$
\hom_{\cC[t]}(a,b)= \{\sum_{i=0}^{n} f_{i}t^{i}: \quad n\in \N\, \quad f_{i}\in \hom_{\cC}(a,b) \}.
$$
Let us also consider the category $\cC[t,t^{-1}]$ with the same objects of $\cC$ and morphisms are
$$
\hom_{\cC[t,t^{-1}]}(a,b)= \{\sum_{i=-n}^{n} f_{i}t^{i}: \quad n\in \N\, \quad f_{i}\in \hom_{\cC}(a,b) \}.
$$
\begin{rem}\label{elremark}
    If $\cC$ is a $\Z$-linear category then $\cC[t]$ and $\cC[t,t^{-1}]$ are $\Z$-linear categories and 
    $$
    \cA{(\cC[t])}\cong \cA(\cC)[t] \qquad \cA{(\cC[t,t^{-1}])}\cong \cA(\cC)[t,t^{-1}].
    $$
\end{rem}
\begin{thm}
    Let $\cC$ be a right regular coherent $\Z$-linear category with finitely many objects then 
    $$
    K_{i}(\cC)\cong K_{i}(\cC[t]) \qquad K_{i+1}(\cC[t,t^{-1}]) \cong K_{i+1}(\cC)\oplus K_{i}(\cC) \qquad i\geq 0.
    $$
\end{thm}
\begin{proof}
    Because \eqref{nemap} is a weak equivalence 
   then  $K_{i}(\cC)\cong K_{i}(\cA(\cC))$. Observe that $\cA(\cC)$ is a right regular coherent ring, by Proposition \ref{catring}. Using \cite[Cor 5.3]{Swan} and \cite[Thm 6.1]{Swan} we obtain 
   $$ K_{i}(\cA(\cC)) \cong  K_{i}(\cA(\cC)[t])$$ for $i\geq 0$. By Remark \ref{elremark} and using again that \eqref{nemap} is a weak equivalence we obtain 
   $$
   K_{i}(\cC)\cong K_{i}(\cA(\cC)) \cong K_{i}(\cA(\cC)[t]) \cong K_{i}(\cA(\cC[t]))\cong K_{i}(\cC[t]) \quad i\geq 0
   $$
   Similarly 
   $$
   \begin{array}{lclr}
   K_{i+1}(\cC[t,t^{-1}])&\cong & K_{i+1}(\cA(\cC[t,t^{-1}])) & \\
   \\
   &\cong & K_{i+1}(\cA(\cC)[t,t^{-1}])\\ 
    \\
   &\cong & K_{i+1}(\cA(\cC))\oplus K_{i}(\cA(\cC)) & \mbox{by \cite[Thm 6.1]{Swan}}
   \\
   \\
   &\cong & K_{i+1}(\cC)\oplus K_{i}(\cC) & i\geq 0
   \end{array}
   $$
\end{proof}

\begin{cor}
    Let $\cC$ be a right AF-regular coherent $\Z$-linear category then 
    $$
    K_{i}(\cC)\cong K_{i}(\cC[t]) \qquad K_{i+1}(\cC[t,t^{-1}]) \cong K_{i+1}(\cC)\oplus K_{i}(\cC) \qquad i\geq 0
    $$
\end{cor}

Denote by $\Z$-$\operatorname{Cat}$ to the category of $\Z$-linear categories.  Consider $\cF$ a full subcategory of $\Z$-$\operatorname{Cat}$.
A functor $F: \Z$-$\operatorname{Cat} \rightarrow \cD$ is $\cF$-{\emph{homotopy invariant}} if $$F(\iota): F(\cC)\rightarrow F(\cC[t])$$ is an isomorphism for every $\cC$ in $\cF$.

\begin{cor}
    Let $\cF$ the full subcategory of right AF-regular coherent $\Z$-linear categories. Then $K_{i}$ is $\cF$-homotopy invariant for $i\geq 0$.
\end{cor}

Using \cite[Thm 3.2]{ep} and Proposition \ref{catring} we obtain the following result.
\begin{prop}\label{generaln}
Let $\cC$ be a $\Z$-linear category with finitely many objects. Suppose that $\cC$ is right regular $n$-coherent. Then 

$$
K_{i}(\cC)\simeq K_{i}(\mathcal{FP}_{n}(\cA(\cC))) \quad i\geq 0.
$$
\end{prop}

\section*{Acknowledgments}

The first author was partially supported by ANII. Both authors were partially
supported by PEDECIBA, CSIC and by the grant ANII FCE-3-2018-1-148588. We
would like to thank Willie Cortiñas and Carlos E. Parra for their interesting comments and discussions. We also thank to the referee for their corrections and comments.

\bibliographystyle{plain}

\end{document}